\documentclass{article}

\usepackage{maa-monthly}
\usepackage[utf8]{inputenc}
\usepackage[T1]{fontenc}
\usepackage{amssymb}
\usepackage[inline]{enumitem}
\usepackage{comment}
\PassOptionsToPackage{hyphens}{url}\usepackage{hyperref}
\usepackage{mathrsfs}
\usepackage{stmaryrd}
\usepackage{tikz-cd}
\usetikzlibrary{arrows}
\usepackage{xcolor}

\definecolor{blue-url}{RGB}{0,0,100}
\definecolor{red-url}{RGB}{100,0,0}
\definecolor{green-url}{RGB}{0,100,0}
\definecolor{light-yellow}{RGB}{255,255,128}
\definecolor{light-blue}{RGB}{193,255,255}
\definecolor{light-red}{RGB}{239,83,80}

\hypersetup{
	pdftitle={On power monoids and their arithmetic},
	pdfauthor={Salvatore Tringali},
	pdfmenubar=false,
	pdffitwindow=true,
	pdfstartview=FitH,
	colorlinks=true,
	linkcolor=blue-url,
	citecolor=green-url,
	urlcolor=red-url
}

\newtheorem{theorem}{Theorem}[section]
\newtheorem{proposition}[theorem]{Proposition}
\newtheorem{corollary}[theorem]{Corollary}
\newtheorem{lemma}[theorem]{Lemma}

\theoremstyle{definition}

\newtheorem{conjecture}[theorem]{Conjecture}
\newtheorem{question}[theorem]{Question}
\newtheorem{questions}[theorem]{Questions}

\providecommand\llb{\llbracket}
\providecommand\rrb{\rrbracket}

\newcommand{\evid}[1]{\textsc{#1}}
\newcommand{\fin}{\mathrm{fin}}
\newcommand{\aut}{\mathrm{Aut}}

\newcommand{\sgrpcong}{ \cong}

\begin{document}

\title{Power monoids and their arithmetic: a survey}
\markright{Power monoids and their arithmetic}
\author{Salvatore Tringali}

\maketitle

\begin{abstract}
The non-empty finite subsets of a multiplicatively written monoid form a monoid under setwise multiplication. The same holds for finite subsets containing the identity element. Partly due to their unusual arithmetic properties, these structures, generically known as power monoids, have attracted increasing attention in recent years, stimulating new perspectives in the study of factorizations in non-cancellative or non-commutative settings. We survey these developments and briefly review some related aspects.
\end{abstract}

\vspace{10px}


\vspace{10px}

\noindent

\section{Introduction.}
\label{sec:01}

Let $S$ be a multiplicatively written semigroup (see Section~\ref{sect:generalities} for notation and terminology). Endowed with the binary operation of setwise multiplication induced by $S$ on its power set and defined by
\[
XY := \{xy : x \in X, \, y \in Y\}, \qquad\text{for all } X, Y \subseteq S,
\]
the \textit{non-empty} subsets of $S$ themselves form a semigroup, denoted by $\mathcal P(S)$. Furthermore, the family of all non-empty \textit{finite} subsets of $S$ forms 
a subsemigroup of $\mathcal P(S)$, denoted by $\mathcal P_\fin(S)$. Obviously, $\mathcal P(S) = \mathcal P_\fin(S)$ if and only if $S$ is finite. We call $\mathcal P(S)$ the \evid{large power semigroup} and $\mathcal P_\fin(S)$ the \evid{finitary power semigroup} of $S$, and refer to either structure generically as a \evid{power semigroup}.

If the semigroup $S$ is written additively, we adopt the same convention for any subsemigroup of $\mathcal P(S)$. This amounts, in practice, to the fact that the operation
on $\mathcal P(S)$ maps a pair $(X,Y)$ of non-empty subsets of $S$ to their
\evid{sumset}
\[
X+Y := \{x+y : x \in X,\; y \in Y\},
\]
as opposed to the setwise product of $X$ by $Y$ used in the multiplicative setting.

Power semigroups provide a natural algebraic framework for a variety of key problems in additive number theory and related areas \cite{Bie-Ger-22, Tri-Yan-23(a), Tri-Yan2025(b), GarSan-Tri-24(a)}, including S\'ark\"ozy's conjecture on the ``additive irreducibility'' of the set of [non-zero] squares of a finite field of prime order \cite[Conjecture 1.6]{Sark2012} and Ostmann's conjecture \cite[p.~13]{Ostm1968} on the ``asymptotic additive irreducibility'' of the set of (positive rational) primes.

The first explicit appearance of power semigroups in the literature dates back to at least a 1950 paper by Ballieu~\cite{Bal1950}. However, their systematic investigation only took off in the late 1960s, continuing intensively through the 1970s--1990s. During this period, research focused primarily on the study of 
\begin{itemize}
\item the lattice of sub[semi]groups of $\mathcal P(G)$ or $\mathcal P_\fin(G)$ when $G$ is a group \cite{McC-Hay-1973, Put-1979, Ped-Siz-1980}; 
\item the automorphisms of the finitary power semigroup of a group \cite{Byr-Llo-Ped-Ste-1977, Byr-Llo-Men-Tel-1978, Byr-Llo-Ste-1982, Byr-Llo-Ped-Ste-1984} (these and related papers appear to have gone
largely unnoticed until April 2026, partly because they use
the term ``semigroups of complexes'' instead of ``power semigroups'');
\item the varieties generated by $\mathcal{P}(S)$ as $S$ ranges over certain families of \textit{finite} semigroups \cite{Pin-1980, Mar-1981, Pin-1995, Alm02}, where a \evid{variety} is a class of semigroups closed under homomorphic images, subsemigroups, and finite direct products. 
\end{itemize}
In particular, these latter developments were largely driven by the role of power semigroups in the theory of formal languages and automata.

Power semigroups have also been central to the ongoing development of an arithmetic theory of semigroups and rings, aiming to extend the classical theory of factorizations~\cite{Ger-Hal-06} beyond its traditional boundaries. Such aspects of the topic form the core of the present survey and will be discussed in much greater detail in Sections~\ref{sect:04} and~\ref{sec:future}. In the remainder of this introduction, we will instead briefly review other directions of research that have contributed to the revival of the subject in the past decade.

Let us just mention here that power semigroups are ``strongly non-cancellative structures'' 
(for instance, if $G$ is a group, then $XG = GX = G$ for every non-empty $X \subseteq \allowbreak G$), and the classical theory of factorizations is ill-suited to handle 
such situations. For this reason, in Section~\ref{sec:3_new_perspectives} we review the basics of an ``extended 
theory of factorizations'', which provides tools to circumvent, or at least mitigate, the side effects arising from these and other obstacles inherent in the classical approach.

\subsection{Isomorphism problems.} A turning point in the history of power semigroups was marked by a paper by Tamura and Shafer \cite{Tam-Sha1967}, eventually leading to the following questions. Here and throughout, the symbol $\sgrpcong$ denotes isomorphism, and every morphism between semigroups is understood to be a semigroup homomorphism.

\begin{questions}
\label{ques:tamura-shafer-iso-problem}
Prove or disprove that, for all $H$ and $K$ in a given class $\mathscr C$ of semigroups, one has
\begin{enumerate}[label=\textup{(\arabic{*})}]
\item\label{ques:tamura-shafer-iso-problem(1)} $\mathcal P(H) \sgrpcong \mathcal P(K)$ if and only if $H \sgrpcong K$.
\item\label{ques:tamura-shafer-iso-problem(2)} $\mathcal P_\fin(H) \sgrpcong \mathcal P_\fin(K)$ if and only if $H \sgrpcong K$.
\end{enumerate}
\end{questions}
The crux of these questions lies in the “only if” direction. Indeed, if $f \colon H \to K$ is a semigroup isomorphism, then the mapping
\[
f^\ast \colon \mathcal P(H) \to \mathcal P(K) \colon X \mapsto f[X], 
\]
where 
\begin{equation}\label{equ:augmentation}
f[X] := \{f(x) \colon x \in X\},
\end{equation}
is itself an isomorphism, called the \evid{augmentation} of $f$. Moreover, the restriction of $f^\ast$ to the non-empty finite subsets of $H$
yields an isomorphism from $\mathcal P_{\fin}(H)$ onto
$\mathcal P_{\fin}(K)$, since clearly $f^\ast(X) \in \mathcal P_\fin(K)$ for all $X \in \mathcal P_\fin(H)$. See \cite[Remark 4]{Tri-2024(a)} and \cite[Section~1]{GarSan-Tri-24(a)} for further details and context.

The answer to Question \ref{ques:tamura-shafer-iso-problem}\ref{ques:tamura-shafer-iso-problem(1)} is negative for the class of all semigroups \cite{Mog1973}; is positive for groups \cite{Shaf-1967}, semilattices \cite[p.~218]{Koba-1984},  Clifford semigroups \cite[Theorem 4.7]{Gan-Zhao-2014}, cancellative commutative semigroups \cite[Corollary 1]{Tri-2024(a)}, etc.;
and is open for finite semigroups \cite[p.~5]{Hami-Nord-2009}, despite some authors having claimed otherwise based on unproven results announced in \cite{Tamu-1987}. For a thorough discussion on this point, see
\begin{center}
\url{https://mathoverflow.net/questions/508548/}
\end{center}

As for Question \ref{ques:tamura-shafer-iso-problem}\ref{ques:tamura-shafer-iso-problem(2)}, very little is known outside the case where $\mathscr C$ is a class of \textit{finite} semigroups contained in any of the classes already covered by the ``positive results'' reviewed in the previous paragraph. 

More precisely, Bienvenu and Geroldinger have shown in \cite[Theorem 3.2(3)]{Bie-Ger-22} that the question has an affirmative answer for numerical monoids, and the conclusion has been generalized to cancellative commutative semigroups in \cite[Corollary 1]{Tri-2024(a)}. Here, a \evid{numerical semigroup} is a cofinite subsemigroup of the non-negative integers under addition, and a \evid{numerical monoid} is a numerical semigroup containing $0$.

\subsection{Power monoids.} Investigations in the past few years have addressed some variants of
Questions~\ref{ques:tamura-shafer-iso-problem}, inspired by the work of Bienvenu
and Geroldinger in \cite{Bie-Ger-22}.

Specifically, let $M$ be a multiplicatively written monoid, with $1_M$ denoting its identity and $M^\times$ its unit group.  $\mathcal P(M)$ itself is then a monoid, and $\mathcal P_\fin(M)$ is a submonoid of $\mathcal P(M)$, their identity being the singleton $\{1_M\}$.
In addition, each of the families
\[
\mathcal P_{\fin,\times}(M) := \{X \in \mathcal P_\fin(M) : 
X \cap M^\times \ne \emptyset\}
\]
and
\[
\mathcal P_{\fin,1}(M) := \{X \in \mathcal P_\fin(M) : 
1_M \in X\} \subseteq \mathcal P_{\fin,\times}(M),
\]
first considered by Fan and Tringali in \cite{Fa-Tr18}, is a submonoid of 
$\mathcal P_\fin(M)$.  
Accordingly, $\mathcal P_\fin(M)$ is called the \evid{finitary power monoid},
$\mathcal P_{\fin,\times}(M)$ the \evid{restricted finitary power monoid}, and
$\mathcal P_{\fin,1}(M)$ the \evid{reduced finitary power monoid} of $M$. We will refer generically to any of these structures as a  \evid{power
monoid}, and write $\mathcal P_{\fin,0}(M)$ instead of
$\mathcal P_{\fin,1}(M)$ when $M$ is written additively.

Recent work on power monoids has mainly focused on $\mathcal P_{\fin,1}(M)$. Note in this regard that $\mathcal P_{\fin,1}(M) = \mathcal P_{\fin,\times}(M)$ if and only if the unit group $M^\times$ of $M$ is trivial, while $\mathcal P_{\fin,\times}(M) = \mathcal P_\fin(M)$ if and only if $M = M^\times$ (that is, $M$ is a group). Consequently, we have that $\mathcal P_{\fin,1}(M) = \mathcal P_\fin(M)$ if and only if $M$ is trivial.

Most notably, Tringali and Yan \cite{Tri-Yan-23(a)} posed the following
analogue of Questions~\ref{ques:tamura-shafer-iso-problem} for $\mathcal P_{\fin,1}(\cdot)$, where, as before,
the symbol $\sgrpcong$ denotes a \textit{semigroup} isomorphism.

\begin{question}\label{ques:BG-like-for-monoids}
Prove or disprove that, for all $H$ and $K$ in a given class of monoids, one has $\mathcal P_{\fin,1}(H) \sgrpcong \mathcal P_{\fin,1}(K)$ if and only if $H \sgrpcong K$.
\end{question}

Once again, the core of Question \ref{ques:BG-like-for-monoids} lies in the ``only if\,'' direction. Indeed, if $f$ is a (semigroup) iso\-mor\-phism from a monoid $H$ to a monoid $K$, then $f$ maps the identity of $H$ to the identity of $K$ (see, for instance, the last lines of \cite[Section~2]{Tri-2024(a)}). Since $f[X]$ is finite for every non-empty finite $X \subseteq H$, it is thus clear that the augmentation of $f$, as defined by Eq.~\eqref{equ:augmentation}, restricts to an isomorphism from $\mathcal P_{\fin,1}(H)$ to $\mathcal P_{\fin,1}(K)$.  

With that said, it was shown in \cite{Tri-Yan-23(a)} that
Question~\ref{ques:BG-like-for-monoids} has a negative answer for left- or right-cancellative monoids (loc.~cit., Examples~1.2), and a
positive answer for (\textsc{rational}) \evid{Puiseux monoids}
(loc.~cit., Theorem~2.5), that is, submonoids of the non-negative rational
numbers under addition. The latter result confirmed a conjecture by Bienvenu and Geroldinger \cite[Conjecture~4.7]{Bie-Ger-22}.

Rago \cite{Rago26} subsequently proved that
Question~\ref{ques:BG-like-for-monoids} has a negative answer within the
class of cancellative commutative monoids. More precisely, he showed that
if $H$ and $K$ are commutative valuation monoids with trivial unit groups
and isomorphic quotient groups, then $\mathcal P_{\fin,1}(H)$ is isomorphic
to $\mathcal P_{\fin,1}(K)$. More recently, he characterized all pairs
$(H,K)$ of non-isomorphic cancellative commutative monoids such that
$\mathcal P_{\fin,1}(H)$ and $\mathcal P_{\fin,1}(K)$ are isomorphic \cite[Theorem 18]{Rago26(c)}. In
particular, he obtained that if $H$ and $K$ have non-trivial unit groups and
$\mathcal P_{\fin,1}(H)$ is isomorphic to $\mathcal P_{\fin,1}(K)$, then
$H$ is isomorphic to $K$ (loc.~cit., Theorem 14).

Meanwhile, Tringali and Yan have established that Question~\ref{ques:BG-like-for-monoids} has an affirmative answer for \textit{torsion} groups \cite{Tri-Yan2026(a)}, and Rago (private communication) has subsequently announced a proof that the same conclusion extends to arbitrary groups.

Finally, one could also consider the companion of Question~\ref{ques:BG-like-for-monoids}, obtained by replacing $\mathcal P_{\fin,1}(\cdot)$ with $\mathcal P_{\fin,\times}(\cdot)$, leading to the following:

\begin{question}\label{ques:BG-like-for-restricted-PMs}
Prove or disprove that, for all $H$ and $K$ in a given class of monoids, one has $\mathcal P_{\fin,\times}(H) \sgrpcong \mathcal P_{\fin,\times}(K)$ if and only if $H \sgrpcong K$.
\end{question}

To date, this has not been studied in the literature, modulo the observation that Questions~\ref{ques:BG-like-for-monoids} and \ref{ques:BG-like-for-restricted-PMs} coincide when attention is restricted to classes of monoids with trivial unit group. As a starter, note that, by Theorem~\ref{thm:divisor-closedness}\ref{thm:divisor-closedness(i)}, the answer is affirmative for groups, since monoid isomorphisms map units bijectively onto units (recall that any \textit{semigroup} isomorphism between monoids is a fortiori a monoid isomorphism, i.e., it maps the identity to the identity).

\subsection{Symmetry and rigidity.} 
The study of automorphisms for a wide range of objects and categories (from topological spaces to fields, from graphs to algebraic varieties, and so on) is a common theme in mathematics and often leads to a deeper understanding of various aspects of the theory. Power semigroups are no exception.

To begin, we gather from the previous subsection that the augmentation of an automorphism $f$ of a semigroup $S$ restricts to an automorphism of the finitary power semigroup $\mathcal P_\fin(S)$. This gives a canonical embedding $\Phi \colon \aut(S) \hookrightarrow \aut(\mathcal P_\fin(S))$, and it is natural to ask whether $\Phi$ is a group isomorphism. Mutatis mutandis, the same question arises for $\aut(\mathcal P_{\fin,1}(H))$ when $H$ is a monoid.

In the aftermath of their proof of the Bienvenu--Geroldinger
conjecture (see the previous subsection for references), Tringali and Yan showed that the automorphism group of
$\mathcal P_{\fin,0}(\mathbb N)$ is cyclic of order two
\cite[Theorem~3.2]{Tri-Yan2025(b)}. The same result had in fact
already been obtained by Byrd et al.~\cite[Theorem~1]{Byr-Llo-Ped-Ste-1984} (via a simpler argument), but their work had gone unnoticed until recently. In addition, it was conjectured in \cite[Section~4]{Tri-Yan2025(b)} that $\aut(\mathcal P_{\fin,0}(H))$ is trivial for every numerical monoid $H$ properly contained in $\mathbb N$; to date, the conjecture is still open.

As a follow-up, Tringali and Wen \cite[Theorem~3.5]{Tri-Wen-2026(a)} showed that the automorphism group of the finitary power monoid of $\mathbb Z$ (the additive group of integers) is isomorphic to the direct product of a cyclic group of order two and the infinite dihedral group, in stark contrast with the fact that the only non-trivial automorphism of $\mathbb Z$ itself is the map $x \mapsto -x$. They also conjectured in \cite[Section~4]{Tri-Wen-2026(a)} that $\aut(\mathcal P_{\fin,0}(\mathbb Z))$ is cyclic of order $2$, and this was confirmed by Wong et al.~in \cite{Wong2026}.

In fact, it had already been established by Byrd et al.~\cite[Theorem~2]{Byr-Llo-Ped-Ste-1984} that the automorphism group of $\mathcal P_\fin(\mathbb Z)$ is a splitting extension of $\mathbb Z$ by the Klein four-group $V_4$; that is, there exists an action $\phi$ of $V_4$ on $\mathbb Z$ such that
$\aut(\mathcal P_\fin(\mathbb Z))$ is isomorphic to the semidirect product $\mathbb Z \rtimes_\phi V_4$. However, since $\aut(\mathbb Z)$ is cyclic of order two, this description does not completely determine the isomorphism type of
$\aut(\mathcal P_{\fin}(\mathbb Z))$, leaving open two
non-isomorphic possibilities for the latter group. An analogous description of $\aut(\mathcal P_\fin(G))$ for every non-trivial additive
subgroup $G$ of the rationals was obtained more generally in \cite[Corollary~4.10]{Byr-Llo-Ste-1982}.

Rago \cite[Theorem~13]{Rago26(b)} subsequently proved that, for a
finite abelian group $G$, the automorphism group of
$\mathcal P_{\fin,1}(G)$ is (canonically) isomorphic to $\aut(G)$,
except when $G \cong V_4$ (the Klein four-group), in which case
$\aut(\mathcal P_{\fin,1}(G))$ is isomorphic to the direct product of
two copies of the symmetric group of degree three
\cite[Example~2]{Rago26(b)}. On a related note, results by Byrd
et al.~\cite[Theorems~2, 3, and~5]{Byr-Llo-Ped-Ste-1977} show that
the automorphism group of $\mathcal P_\fin(G)$, when $G$ is either a
finite cyclic group of order $n \notin \{3,4,5\}$ or a subgroup of
the circle group of order greater than $5$, is (canonically)
isomorphic to $\aut(G)$. The same paper also contains a description
of $\aut(\mathcal P_\fin(G))$ when $G$ is cyclic of order $3$, $4$,
or $5$ (ibid., Sect.~3, p.~31).

Even more recently, Wong et al.~\cite[Theorem~1.3]{Wong2025} have proved that the automorphism group of the finitary power semigroup $\mathcal{P}_{\text{fin}}(S)$ of a numerical semigroup $S$ is trivial unless $S = \llbracket k, \infty \llbracket$ for some integer $k \ge 0$, in which case
$\operatorname{Aut}(\mathcal{P}_{\text{fin}}(S))$ is cyclic of order two. 
This result has been a source of inspiration for~\cite{Tri-Wen-2026(b)}, where it is established that the automorphism groups of $\mathcal{P}(S)$ and $\mathcal{P}_0(H)$ are trivial for all numerical semigroups $S$ and numerical monoids $H$. Here, $\mathcal{P}_0(H)$ is the submonoid of $\mathcal{P}(H)$ 
consisting of all subsets of $H$ containing $0$.

 \section{Formalities.}
\label{sect:generalities}

We denote by $\mathbb N$ the additive monoid of non-negative integers, by $\mathbb N^+$ the set of positive integers, by $\mathbb Z$ the additive group of integers, and by $|X|$ the cardinality of a set $X$.
Unless otherwise stated, we reserve the letters $m$ and $n$ (with or without subscripts) for positive integers.
Given $a, b \in \mathbb N \cup \{\infty\}$, we let $\llb a, b \rrb := \{x \in \allowbreak \mathbb N \colon \allowbreak a \le x \le b\}$ be the (\evid{discrete}) \evid{interval} from $a$ to $b$.

If not explicitly specified, we write all semigroups (and monoids) multiplicatively. An element $a$ in a semigroup $S$ is \evid{cancellative} if $ax \ne ay$ and $xa \ne ya$ for all $x, y \in S$ with $x \ne y$; the semigroup itself is cancellative if each of its elements is.

A \evid{unit} of a monoid $M$ with identity (element) $1_M$ is an element $u \in M$ for which there exists a (provably unique) element $v \in M$, accordingly denoted by $u^{-1}$ and called the \evid{inverse} of $u$ (in $M$), with the property that $uv = vu = 1_M$. We assume that every submonoid of a monoid is \evid{unitary}, meaning that it has the same identity.

The monoid $M$ is \evid{Dedekind-finite} if, for any $x, y \in M$, the equality $xy = 1_M$ implies $yx = 1_M$; \evid{unit-cancellative} if $xy \ne x \ne yx$ for all $x, y \in M$ with $y$ a non-unit; \evid{acyclic} if $uxv \ne x$ for all $u, v, x \in M$ such that $u$ or $v$ is a non-unit; and \evid{torsion-free} if the only finite-order element of $M$ is the identity. Here and throughout, the \evid{order} of an element $x \in M$ is the size of the set $\{x^k : \allowbreak k \in \allowbreak \mathbb N\}$ if this set is finite;
otherwise, the order is $\infty$.

To the best of our knowledge, unit-cancellativity was first studied in the commutative setting by Rosales et al.~\cite[Section~4, p.~72]{Ros-GS} (under a different name), and in the non-commutative setting by Fan and Tringali \cite[Section~2.1]{Fa-Tr18}. Acyclicity, on the other hand, was first introduced in \cite[Definition~4.2]{Tr20(c)}. Commutative or
unit-cancellative monoids are Dedekind-finite \cite[Proposition 4.3(i)]{Tr20(c)}. Moreover, every acyclic monoid is unit-cancellative, and the converse holds in the commutative setting.

Given a set $A$, we denote by $\mathscr F^+(A)$ the \evid{free semigroup over $A$}
and write $\ast_A$ for its operation; when no ambiguity is likely to arise, we drop the subscript $A$
from this notation.
As a set, $\mathscr F^+(A)$ consists of all non-empty finite tuples of
elements of $A$, and $\ast_A$ is the operation of \evid{concatenation}
of such tuples. Adding the empty tuple as the identity element for concatenation yields the \evid{free monoid} $\mathscr F(A)$ over $A$.

We refer to an element of $\mathscr F(A)$ as an \evid{$A$-word}, or simply a \evid{word} when there is no serious risk of confusion. The empty tuple is called the empty $A$-word.
If $\mathfrak u$ is a non-empty $A$-word, its \evid{length} is the unique integer $k \ge 1$ such that
$\mathfrak u$ is a $k$-tuple of elements of $A$. In this case, $
\mathfrak u = u_1 \ast \cdots \ast u_k$
for some uniquely determined $u_1, \ldots, u_k \in A$, called the
\evid{letters} of $\mathfrak u$. The empty $A$-word has length zero.

Further notation and ter\-mi\-nol\-o\-gy, if not explained when first used, are standard or should be clear from the context. In particular, we address the reader to Howie's monograph \cite{Ho95} for basic aspects of semigroup theory.
\section{New perspectives on an old classic.}
\label{sec:3_new_perspectives}

Throughout this section, $M$ denotes a monoid, with no further
assumptions unless otherwise specified.

Roughly speaking, the theory of factorizations is the study of a broad spectrum of phenomena arising from the possibility or impossibility of extending the Fundamental Theorem of Arithmetic from the integers to more abstract settings where a set comes endowed with a binary operation.
The theory has traditionally focused on domains and cancellative monoids,
primarily in the commutative setting
\cite{Ger-Hal-06, Bag-Cha-2011, Ger-2016, Ge-Zh-20a, Got-And-2022, Sme-2026}, and the
building blocks used in factorizations are \evid{atoms}, that is,
non-units that cannot be written as a product of two other non-units.
Only in the past few years have researchers taken the first
steps toward a systematic study of factorizations in
non-cancellative (and non-commutative) settings, using building blocks
that are not necessarily atoms; see \cite[Section~2.4]{An-Tr18} for a
review of some earlier scattered results along these lines. 

Power monoids are among the driving
forces behind these developments: they exhibit a rich arithmetic, which makes them an ideal testing ground for some key aspects of an ``extended theory of factorizations'' whose foundations were laid down by Tringali et al.~in \cite{Fa-Tr18, An-Tr18, Tr20(c), Tr21(b), Co-Tr-21(a), Co-Tr-22(a), Co-Tr-22(b)}. See \cite{Ajr-Got-2023pre, Cas-DAnn-GarSan-2023, Bon-GarSan-2024, GarSan-ONei-2026} for applications to Boolean lattices, monoids of ideals, and semilattices of numerical monoids.

A key feature of these new perspectives is the idea of pairing the
monoid $M$ with a \textsc{preorder} (that is, a reflexive and transitive binary
relation on $M$) and letting the preorder induce on $M$ a
\emph{relative} notion of irreducibility \cite[Section~3]{Tr20(c)}.
In this way, the theory gains tremendous flexibility, as neatly
explained by Cossu in \cite{Cossu}, allowing, for instance, the study
of factorizations in non-trivial contexts such as groups, where the
classical theory is inapplicable because there are no atoms, or in rings
with non-zero zero divisors, where the classical theory loses
significance as its invariants tend to blow up in a more or less
predictable way. 

Another feature of this extended theory, specifically designed
to counter, or at least mitigate, the aforementioned blow-up
phenomena, is the idea of imposing a ``minimality condition'' on
factorizations, an idea first introduced in
\cite[Section~4]{An-Tr18} within the restricted framework of the
classical theory and further developed in
\cite{Co-Tr-22(a), Co-Tr-22(b), Co-Tr-25(a)}.

Below, we briefly review the basics of this generalized approach that are most relevant to the study of power monoids in the subsequent sections.

\subsection{Irreducibles vs atoms.} 

To avoid unnecessary technicalities, we shall not pursue the broader
generality alluded to above, but restrict our attention to the
\evid{divisibility preorder} $\mid_M$, that is, the binary relation on $M$ defined, for all $x, y \in M$, by
\[
x \mid_M y \quad \text{if and only if} \quad
y \in MxM.
\]
If $x \mid_M y$, we say that $x$ is a \evid{divisor} of $y$ (in $M$),
or $x$ divides $y$, or $y$ is divisible by $x$. If $x \mid_M y$ and $y \mid_M x$, we write $x \simeq_M y$. If $x$ does \textit{not} divide $y$, we write $x \nmid_M y$. We call
$x$ a \evid{unit-divisor} if $x$ is a divisor of the identity $1_M$, and a \evid{proper divisor} of $y$ if $x \mid_M y$ but $y \nmid_M x$. For later reference, it is convenient to set
\[
M^\sharp := \{x \in M: x \nmid_M 1_M\}.
\]
A submonoid $N$ of $M$ is \evid{divisor-closed} if every divisor in $M$ 
of an element $a \in N$ is itself an element of $N$; in this case, we may equally say that $N$ is divisor-closed in $M$.  
Note that every unit divides every element of $M$, and $M$ is
Dedekind-finite if and only if $M^\sharp = M \setminus M^\times$ (that is, every unit-divisor is a unit).

Within this framework, a key role is played by the \evid{$\mid_M$-irreducibles}, which we simply call the \evid{irreducibles}, or \evid{irreducible elements}, of $M$. These are the elements $a \in M^\sharp$
that admit no factorization $a = bc$ with
$b$ and $c$ proper divisors of $a$ in $M$ that are not unit-divisors. In particular, it follows from the previous observations that, in a Dedekind-finite monoid, an irreducible is a non-unit that does not factor as a product of two proper divisors.
In addition to irreducibles, we shall occasionally encounter
\evid{quarks}, that is, elements whose only proper divisors divide the identity, yet which are not themselves unit-divisors.

Atoms (in the sense of the classical theory) and quarks are
irreducibles, but the converse need not be true. 
For instance, the zero element $0_R$ of a domain $R$ is irreducible in the multiplicative monoid of $R$, yet it is not an atom, since it is idempotent. Moreover, if $R$ is not a skew field, then $0_R$ is not a quark either,
because every non-unit $x \in R$ divides $0_R$. For a more interesting
example, see Theorem~\ref{thm:Pfin1-is-factorable}.

\begin{proposition}
\label{prop:atoms-are-quarks}
Every atom of a monoid is a quark.
\end{proposition}

\begin{proof}
Let $a$ be an atom of $M$ (if $M$ has no atoms, the statement is vacuously true), and let $b \in M$ be a divisor of $a$ but not a unit-divisor. By \cite[Lemma 2.2(i)]{Fa-Tr18}, $M$ is Dedekind-finite, and hence $b$ is a non-unit. We need to show that $a \mid_M b$.

Since $b \mid_M a$, there exist $u, v \in M$ such that $a = ubv$; and since $a$ is an atom, this can only happen if each of the sets $\{u,bv\}$ and $\{ub, v\}$ contains a unit. By the Dedekind-finiteness of $M$ and the fact
that $b \notin M^\times$, it follows that both $u$ and $v$ are units. Hence $b = u^{-1} a v^{-1}$, that is, $a$ divides $b$ (as desired).
\end{proof}

For the reader's convenience, we summarize the mutual relationships between
irreducibles, atoms, and quarks in the
following diagram:
\[
\{\text{atoms}\} \subseteq \{\text{quarks}\} \subseteq \{\text{irreducibles}\}.
\]

It should be noted that our distinction between irreducibles and atoms is a delicate point and a
common source of confusion. While in the classical theory the two terms are used interchangeably, the former notion is considerably more flexible than the latter and better suited to the extended
version of the theory that we consider here; see, e.g., \cite[Theorem~3.10]{Tr20(c)},
\cite[Theorem~5.19]{Co-Tr-21(a)}, and \cite[Theorem~5.1]{Co-Tr-22(a)}.

\subsection{Factorizations and lengths.} The monoid $M$ is \evid{factorable} (resp., \evid{atomic}) if every $x \in M^\sharp$ can be written as a product $a_1 \cdots a_m$ of irreducibles (resp., atoms), in which case the word $a_1 \ast \allowbreak \cdots \ast a_m$ (an element of the free semigroup over $M$) is called a \evid{factorization} (resp., an \evid{atomic factorization}) of $x$ (over $M$) and $m$ is called a \evid{length} (resp., an \evid{atomic length}) of $x$. 
We denote by $\mathscr{Z}_M(x)$ the set of factorizations of the element $x$, and by $\mathsf{L}_M(x)$ its set of lengths.

A factorization $\mathfrak a \in \mathscr Z_M(x)$ is \evid{equivalent} to a factorization $\mathfrak b \in \mathscr Z_M(x)$ if the length $m$ of $\mathfrak a$ is equal to the length $n$ of $\mathfrak b$ and there exists a permutation $\sigma$ of the interval $\llb 1, m \rrb$ such that $a_i \simeq_M b_{\sigma(i)}$ for every $i \in \llb 1, m \rrb$, where $a_i$ and $b_{\sigma(i)}$ are the $i^\text{th}$ letter of $\mathfrak a$ and the $\sigma(i)^\text{th}$ letter of $\mathfrak b$. Of course, ``being equivalent'' is an equivalence relation on $\mathscr Z_M(x)$. In addition, $\mathfrak a$ is a \evid{minimal factorization} (of $x$) if there do not exist a factorization $\alpha_1 \ast \cdots \ast \alpha_k \in \mathscr Z_M(x)$ of length $k$ (strictly) smaller than $m$ and an injective map $\tau \colon \llb 1, k \rrb \to \llb 1, m \rrb$ such that $\alpha_i \simeq_M a_{\tau(i)}$ for each $i \in \llb 1, k \rrb$, in which case $m$ is called a \evid{minimal length} of $x$. We denote by $\mathsf L_M^{\sf m}(x)$ the set of minimal lengths and by $\mathscr Z_M^{\sf m}(x)$ the set of minimal factorizations of $x$. It is fairly obvious that $x$ has a factorization if and only if it has a minimal factorization; moreover, 
\[
\mathsf L_M^{\sf m}(x) \subseteq \mathsf L_M(x)
\quad\text{and}\quad
\mathscr Z_M^{\sf m}(x) \subseteq \mathscr Z_M(x).
\]

For convenience, we extend some of the previous definitions, by setting $\mathsf L_M(1_M) = \mathsf L_M^{\sf m}(1_M)  := \{0\} \subseteq \mathbb N$ and
$\mathsf L_M(u) = \mathsf L_M^{\sf m}(u) := \emptyset$ for every non-identity unit-divisor $u \in M$. (Note that a non-empty product of irreducibles cannot be a unit-divisor.)
Accordingly, it is natural to consider the families
\[
\mathscr L(M) := \{\mathsf L_M(y) : y \in M\} \setminus \{\emptyset\}
\]
and
\[
\mathscr L^{\sf m}(M) := \{\mathsf L_M^{\sf m}(y) : y \in M\} \setminus \{\emptyset\}.
\]
These are subfamilies of the power set of $\mathbb N$, and both contain the
singleton $\{0\}$. Moreover, our definitions ensure that if $0 \in \mathsf
L_M(y)$ for some $y \in M$, then $y$ is the identity $1_M$, and hence
$\mathsf L_M(y) = \mathsf L_M^{\sf m}(y) = \{0\}$.

Each $L \in \mathscr L(M)$ is called a \evid{length set} of $M$, and
$\mathscr L(M)$ the \evid{system of lengths} of the monoid. In a similar vein, each $L \in \mathscr L^{\sf m}(M)$ is called a \evid{minimal length set} of $M$, and
$\mathscr L^{\sf m}(M)$ the \evid{system of minimal lengths}.

A fundamental tenet of factorization theory is that, apart from trivial cases, the arithmetic of a monoid $M$ is controlled by the arithmetic of its \textit{proper} divisor-closed submonoids. From \cite[Corollary 3.6]{Co-Tr-22(a)}, we know that if $N$ is a divisor-closed 
submonoid of $M$, then $N$ and $M$ share the same unit-divisors, and every irreducible 
(resp., atom) of $N$ is also an irreducible (resp., atom) of $M$. We are thus led to the following proposition, which is a special case of \cite[Proposition 3.5]{Co-Tr-22(a)}.

\begin{proposition}
\label{prop:divisor-closed-submons}
If $N$ is a divisor-closed submonoid of the monoid $M$, then every [minimal] factorization of an element $x \in N$ is also a [minimal] factorization of $x$ in $M$. In particular, the system of [minimal] lengths of $N$ is contained in that of $M$.
\end{proposition}

Length sets are among the most studied arithmetic invariants in the
classical theory of factorizations. In this regard, it is worth noting
that, for an acyclic monoid, every irreducible is an
atom \cite[Corollary 4.4]{Tr20(c)}, so there is no distinction between factorizations and atomic
factorizations. 

We say that $M$ is \evid{factorial} (shortly, \evid{UF}) if it is factorable and any two factorizations of an element in $M^\sharp$ are equivalent; \evid{finite-factorial} (shortly, \evid{FF}) if it is factorable and, for each  $x \in M^\sharp$,  the set of factorizations of $x$ is finite up to equivalence; \evid{half-factorial} (shortly, \evid{HF}) if it is factorable and every length set is a singleton; and \evid{bounded-factorial} (shortly, \evid{BF}) if it is factorable and every length set is finite.

These definitions are modeled on their classical counterparts \cite[Theorem 1.1.10.2 and Definitions 1.1.1.4, 1.2.6, 1.3.1, and 1.5.1]{Ger-Hal-06}, and it is fairly obvious that the implications shown in the following diagram hold:
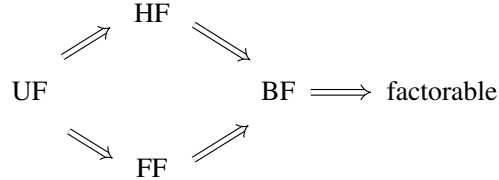
\[
\begin{tikzcd}[arrows=Rightarrow, row sep=0.5cm]
	&
	\!\text{HF}\arrow[shorten <= 2pt,shorten >= 2pt, shift right=-1ex]{dr}
	&
	\\
	\text{UF}
	\arrow[shorten <= 4pt, shorten >= 4pt, shift right=-1ex]{ur}
	\arrow[shorten <= 7pt, shorten >= 4pt, shift left=-1ex]{dr}
	&  &
	\text{BF} \arrow[shorten <= 1pt, shorten >= 1pt]{r}
	&
	\text{factorable}  \\
	&
	\!\text{FF} \arrow[shorten <= 3pt,shorten >= 3pt, shift left=-1ex]{ur}
	&
\end{tikzcd}
\] 

Analogous definitions can be formulated for minimal factorizations. More explicitly,
$M$ is \evid{UmF} if it is factorable and any two minimal factorizations of an element in $M^\sharp$ are equivalent; \evid{FMF} if it is factorable and, for each $x \in M^\sharp$, the set of minimal factorizations of $x$ is finite up to equivalence; and
\evid{HmF} (resp., \evid{BmF}) if it is factorable and every minimal length set is a singleton (resp., is finite).

Lastly, restricting to atomic factorizations yields the
corresponding notions of \evid{UF-atomic}, \ldots, \evid{BF-atomic},
\evid{UmF-atomic}, \ldots, and \evid{BmF-atomic} monoid.

\section{Arithmetic in power monoids.}
\label{sect:04}

Throughout, $H$ denotes a monoid, and no additional assumptions on $H$
are made unless explicitly stated.

In this section, we survey several results concerning 
irreducible, atomic, and minimal factorizations in $\mathcal P_{\fin,1}(H)$ and $\mathcal P_{\fin,\times}(H)$, and also prove a number of new ones. 
The finitary power monoid $\mathcal P_\fin(H)$ of $H$ will be mentioned only occasionally, since the current 
state of knowledge does not yet allow for a systematic treatment comparable 
to that of $\mathcal P_{\fin,1}(H)$ and $\mathcal P_{\fin,\times}(H)$.
Some partial results, notably by Fan and Tringali \cite{Fa-Tr18}, Antoniou 
and Tringali \cite{An-Tr18}, Bienvenu and Geroldinger \cite{Bie-Ger-22}, 
Gonzalez et al.~\cite{Gon-Li-Rab-Rod-Tir-2025}, and Gotti et al.~\cite{
Agg-Got-Su-2024, Dan-Got-Hon-Li-Sch-2025}, will not be discussed here.

\subsection{Overture.} There are many arguments supporting the claim that  $\mathcal P_{\fin,1}(H)$ is, in a sense, tamer than other power monoids. One basic point to keep in mind is that 
\begin{equation}
\label{eq:to-divide-implies-to-be-contained}
\text{if } X \text{ divides } Y \text{ in } \mathcal P_{\fin,1}(H), \text{ then } X \subseteq Y.
\end{equation}
Another elementary observation (see, e.g., Proposition 3.2(ii) in \cite{An-Tr18}) is that
\begin{equation}
\label{equ:increasing-size}
|XY| \ge \max(|X|, |Y|),\qquad \text{for all }
X, Y \in \mathcal P_{\mathrm{fin},\times}(H).
\end{equation}

Eqs.~\eqref{eq:to-divide-implies-to-be-contained} and \eqref{equ:increasing-size} enter into the proof of the next two results, where a monoid $M$ is said to satisfy the \evid{ascending chain condition} (for short, ACC) \evid{on principal two-sided ideals} if there is no sequence 
$x_1, \allowbreak x_2, \dots$ in $M$ such that $M x_n M \subsetneq Mx_{n+1} M$ (or equivalently, $x_{n+1}$ is a proper divisor of $x_n$) for all $n$. 

\begin{theorem}
\label{thm:ACC-and-Dedekind-finiteness}
The following hold:

\begin{enumerate}[label=\textup{(\roman{*})}]
\item\label{thm:ACC-and-Dedekind-finiteness(i)} $\mathcal P_{\mathrm{fin},1}(H)$ is a Dedekind-finite monoid with trivial 
unit group, satisfies the ACC on principal two-sided ideals, and is factorable.

\item\label{thm:ACC-and-Dedekind-finiteness(ii)} $\mathcal P_{\fin,1}(K)$ is a divisor-closed submonoid of $\mathcal P_{\fin,1}(H)$ for every submonoid $K$ of $H$ (regardless of whether $K$ is divisor-closed in $H$).
\end{enumerate}
\end{theorem}

\begin{proof}
See \cite[Proposition 4.11(i)--(ii)]{Tr20(c)} for \ref{thm:ACC-and-Dedekind-finiteness(i)}, and \cite[Proposition 3.2(iii)]{An-Tr18} for \ref{thm:ACC-and-Dedekind-finiteness(ii)}.
\end{proof}

\begin{theorem}
\label{thm:divisor-closedness}
The following hold:
\begin{enumerate}[label=\textup{(\roman{*})}]
\item\label{thm:divisor-closedness(i)} The units of $\mathcal P_{\fin,\times}(H)$ and $\mathcal P_\fin(H)$ are the singletons $\{u\}$ with $u \in H^\times$.

\item\label{thm:divisor-closedness(ii)} $\mathcal P_{\fin,\times}(H)$ is divisor-closed in $\mathcal P_\fin(H)$ if and only if $H$ is Dedekind-finite.

\item\label{thm:divisor-closedness(iii)} $\mathcal P_{\fin,\times}(H)$ is Dedekind-finite (regardless of any conditions on $H$),
whereas $\mathcal P_{\fin}(H)$ is Dedekind-finite if and only if $H$ is.

\item\label{thm:divisor-closedness(iv)} $\mathcal P_{\fin,\times}(H)$ satisfies the ACC on principal two-sided ideals and is factorable.
\end{enumerate}
\end{theorem}

\begin{proof}
For \ref{thm:divisor-closedness(i)}, see \cite[Proposition 3.2(iv)]{An-Tr18}. 
For \ref{thm:divisor-closedness(ii)}, use that $H$ embeds into $\mathcal P_\fin(H)$ via the homomorphism $x \mapsto \{x\}$, and recall that a monoid is Dedekind-finite if and only if a 
non-empty product is a unit precisely when each factor is a unit.

\ref{thm:divisor-closedness(iii)} The Dedekind-finiteness of $\mathcal P_{\fin,\times}(H)$ follows directly from \ref{thm:divisor-closedness(i)} and Eq.~\eqref{equ:increasing-size}. 
Moreover, for $\mathcal P_\fin(H)$ to be Dedekind-finite it is clearly necessary that $H$ is, since $H$ embeds into $\mathcal P_\fin(H)$ via the homomorphism $x \mapsto \{x\}$.

Now to finish the proof assume that $H$ is Dedekind-finite. If $XY = \{1_H\}$ for some $X, Y \in \mathcal P_\fin(H)$, then each of $X$ and $Y$ contains a unit of $H$, namely, $X, Y \in \mathcal P_{\fin,\times}(H)$. This implies from the previous part that $X$ and $Y$ are units of $\mathcal P_\fin(H)$, which shows in turn that $\mathcal P_\fin(H)$ is Dedekind-finite.

\ref{thm:divisor-closedness(iv)} Suppose, for a contradiction, that $\mathcal P_{\fin,\times}(H)$ does not satisfy the ascending chain condition on principal two-sided ideals (ACCP, for short).
There then exists a sequence $X_1, X_2, \dots$ in $\mathcal P_{\fin,\times}(H)$ such that $X_{n+1}$ is a \textit{proper} divisor of $X_n$ for each $n$.
In particular, there exist $u_n, v_n \in H^\times$ with 
$u_n X_{n+1} v_n \subseteq X_n$, and this inclusion must be strict;
otherwise, $X_{n+1} = u_n^{-1} X_n v_n^{-1}$, making $X_n$ a divisor of $X_{n+1}$. It follows that $
|X_{n+1}| = |u_n X_{n+1} v_n| < |X_n|$ for each $n$,
which is impossible, as it would result in a strictly decreasing (infinite) sequence of positive integers.

Thus, $\mathcal P_{\fin,\times}(H)$ satisfies the ACCP,
which completes the proof, since $\mathcal P_{\fin,\times}(H)$ is Dedekind-finite by \ref{thm:divisor-closedness(iii)}
and every Dedekind-finite monoid satisfying the ACCP is known to be factorable
\cite[Theorem~3.10]{Tr20(c)}.
\end{proof}

\subsection{Interlude on cancellation.} We have already mentioned that power monoids are far from being cancellative, which is one reason why much of the classical machinery of factorization theory breaks down when applied to the study of their arithmetic. The next results make this statement slightly more precise.

\begin{proposition}
\label{prop:cancellativity}
The following conditions are equivalent:
\begin{enumerate}[label=\textup{(\alph{*})}]
\item\label{prop:cancellativity(i)} $H$ is trivial.
\item\label{prop:cancellativity(ii)} $\mathcal P_{\fin}(H)$ is cancellative.
\item\label{prop:cancellativity(iii)} $\mathcal P_{\fin,\times}(H)$ is cancellative.
\item\label{prop:cancellativity(iv)} $\mathcal P_{\fin,1}(H)$ is cancellative.
\end{enumerate}
\end{proposition}

\begin{proof}
Since every subsemigroup of a cancellative semigroup is cancellative, it is obvious that \ref{prop:cancellativity(i)} $\Rightarrow$ \ref{prop:cancellativity(ii)} $\Rightarrow$ \ref{prop:cancellativity(iii)} $\Rightarrow$ \ref{prop:cancellativity(iv)}. It remains to see that \ref{prop:cancellativity(iv)} $\Rightarrow$ \ref{prop:cancellativity(i)}.

If $H$ contains an element $x$ of order greater than $2$, then $\{1_H, x\}^2 \ne \{1_H, \allowbreak x^2\}$, yet 
$\{1_H, x\}^3 = \{1_H, x\} \{1_H, x^2\}$.  
If, on the other hand, $H$ has an element $y$ of order $2$, then 
$\{1_H\} \ne \{1_H, y\} = \{1_H, y\}^2$.  
In either case, $\mathcal P_{\fin,1}(H)$ is not cancellative.
\end{proof}

Although cancellativity fails dramatically in power monoids, it is still possible to guarantee certain ``surrogates'' (such as unit-cancellativity or torsion-freeness) under assumptions that, while strong, hold in many situations of interest.

\begin{proposition}
\label{prop:unit-cancellativity}
If $H$ is a submonoid of a torsion-free group, then $\mathcal P_{\fin,1}(H)$ is acyclic and, hence, unit-cancellative and torsion-free.
\end{proposition}

\begin{proof}
Obviously, every acyclic monoid is unit-cancellative. Moreover, it is clear that any acyclic monoid $M$ with trivial unit group is torsion-free: if $x^m = x^n$ for some $x \in M$ and $m, n \in \mathbb N^+$ with $m < n$, then $x$ must be a unit and hence the identity. By Theorem \ref{thm:ACC-and-Dedekind-finiteness}\ref{thm:ACC-and-Dedekind-finiteness(i)}, it is therefore sufficient to show that $\mathcal P_{\fin,1}(H)$ is acyclic.

Suppose that $AXB = X$ for some $A, B, X \in \mathcal P_{\fin,1}(H)$. Since $H$ is contained in a torsion-free group, it follows from Kemperman's inequality (see the unnumbered corollary to Theorem 5 on p.~251 of \cite{Kem-1956}) that 
\[
|X| = |AXB| \ge |A| + |X| + |B| - 2.
\]
Consequently, we must have $|A| + |B| \le 2$, which is only possible if $|A| = |B| = 1$, that is, $A = B = \{1_H\}$. This shows that $\mathcal P_{\fin,1}(H)$ is acyclic (as desired).
\end{proof}

For the next proposition, we recall that a semigroup $S$ is \evid{linearly orderable} 
if there exists a total order $\preceq$ on $S$ such that 
$x \prec y$ (that is, $x \preceq y$ and $x \ne y$) implies 
$uxv \prec uyv$ for all $u, v \in S$. 
For instance, any cancellative commutative torsion-free monoid is linearly orderable, 
and many more examples are listed in \cite{Tri-2015}.

We also say that a monoid $M$ is \textsc{strongly unit-cancellative} 
if $xy \ne x$ and $yx \ne x$ for all $x, y \in M$ with $y \ne 1_M$. 
Of course, every strongly unit-cancellative monoid is unit-cancellative, 
and the converse holds when the unit group is trivial.

\begin{proposition}
If $H$ is a linearly orderable monoid, then $\mathcal P_\fin(H)$ and $\mathcal P_{\fin,\times}(H)$ are strongly unit-cancellative and torsion-free.
\end{proposition}

\begin{proof}
For a proof that $\mathcal P_\fin(H)$ is strongly unit-cancellative, see Proposition 3.5(ii) in \cite{Fa-Tr18}.
Since the properties of being strongly unit-cancellative or torsion-free pass to submonoids, it remains to show that $\mathcal P_\fin(H)$ is torsion-free.

To this end, let $\preceq$ be a total order on $H$ such that,
whenever $x \prec y$, one has $uxv \prec uyv$
for all $u, v \in H$. 
Every non-empty finite subset $X$ of $H$ has a $\preceq$-minimum and a $\preceq$-maximum, i.e., there exist $\alpha, \beta \in X$ such that
\[
\alpha \preceq x \preceq \beta,
\qquad\text{for all }x \in X. 
\]

Now, suppose $X \ne \{1_H\}$. It is clear that either $\alpha \prec 1_H$ or $
1_H \prec \beta$, and we may assume without loss of generality that the latter holds (the other case being symmetric).
If $m < n$, then $
\max X^m = \beta^m \prec \beta^n = \max X^n$,
and hence $X^m \ne X^n$. 

Since $X$ is an arbitrary non-identity element of $\mathcal{P}_{\text{fin}}(H)$, 
we conclude that $\mathcal{P}_{\text{fin}}(H)$ is indeed torsion-free.
\end{proof}

\subsection{Minuetto of irreducibles.} We now turn to clarifying the relationship between irreducibles, atoms, and quarks.
We begin with a key property of two-element sets.

\begin{proposition}
\label{prop:two-element-irreds}
$\{1_H, x\}$ is a quark of $\mathcal P_{\fin,1}(H)$ for each non-identity $x \in H$, while $\{u, x\}$ is a quark of $\mathcal P_{\fin,\times}(H)$ for all $u \in H^\times$ and $x \in H$ with $u \ne x$. 
\end{proposition}

\begin{proof}
The first part is straightforward from Eq.~\eqref{eq:to-divide-implies-to-be-contained}, so we focus on the second.

Let $A \in \mathcal P_{\fin,\times}(H)$ be a two-element set, and suppose that $X$ is a divisor of $A$ in $\mathcal P_{\fin,\times}(H)$ but not a unit. By Theorem \ref{thm:divisor-closedness}\ref{thm:divisor-closedness(i)}, $X$ has at least two elements, and there exist $S, T \in \mathcal P_{\fin,\times}(H)$ such that $A = SXT$. Pick a unit $u \in S$ and a unit $v \in T$. Since $uXv \subseteq A$ and $2 \le |X| = |uXv| = |A| = 2$, it follows that $A = uXv$, and hence $X = \allowbreak u^{-1} Av^{-1}$. Thus, $A$ divides $X$, thereby proving that $A$ is a quark.
\end{proof}

We continue with a characterization of quarks and atomicity in $\mathcal P_{\fin,1}(H)$.

\begin{theorem}\label{thm:Pfin1-is-factorable}
The following hold in $\mathcal P_{\fin,1}(H)$:
\begin{enumerate}[label=\textup{(\roman{*})}]
\item\label{thm:Pfin1-is-factorable(i)} Every irreducible is a quark.
\item\label{thm:Pfin1-is-factorable(ii)} Every irreducible is an atom, with the exception of the two-element sets $\{1_H, x\} \subseteq H$ 
for which $x^2 = 1_H$ or $x^2 = x$.
\end{enumerate}
\end{theorem}

\begin{proof}
For \ref{thm:Pfin1-is-factorable(i)}, see \cite[Proposition 4.11(iii)]{Tr20(c)}. For \ref{thm:Pfin1-is-factorable(ii)}, see
\cite[Proposition~2.3]{Co-Tr-25(a)}, where it is however mistakenly stated (the proof is correct)
that $X$ is irreducible but not an atom ``if and only if
$X = \{1_H, x\}$ for some $x \in H$ such that $x^2 = 1_H$ or $x^2 = x$''
(one should add that $x \ne 1_H$).
\end{proof}

\begin{corollary}
\label{cor:atomicity}
$\mathcal P_{\fin,1}(H)$ is atomic if and only if $H$ has no unit of order two and its only idempotent is the identity $1_H$.
\end{corollary}

\begin{proof}
Apply Theorems \ref{thm:ACC-and-Dedekind-finiteness}\ref{thm:ACC-and-Dedekind-finiteness(i)} and \ref{thm:Pfin1-is-factorable}\ref{thm:Pfin1-is-factorable(ii)}, upon observing that, for every non-identity $x \in H$, the set $\{1_H, x\}$ factors into a product of atoms if and only if it is an atom.
\end{proof}

Now we look for analogous characterizations in $\mathcal P_{\fin,\times}(H)$.

\begin{lemma}
\label{lem:large-irreducibles-in-Pfin_times}
Assume $H$ is Dedekind-finite, and let $A, X, Y \in \mathcal P_{\fin,\times}(H)$ satisfy $A = XAY$. If $|A| \ge 3$ and $\max(|X|, |Y|) \ge 2$, then $A$ is not irreducible in $\mathcal P_{\fin,\times}(H)$.
\end{lemma}

\begin{proof}
Assume without loss of generality that $|X| \ge 2$ (the case $|Y| \ge 2$ is symmetric). Pick a unit $a \in A$. By the hypothesis that $A = XAY$, there exist $u \in X$, $b \in A$, and $v \in Y$ such that $a = ubv$, which, by the Dedekind-finiteness of $H$, is only possible if $u, v \in H^\times$. Since $uAv \subseteq A$ and $|uAv| = |A|$, it follows that $A = uAv$. 

Now, using that $|X| \ge 2$, pick an element $x \in X$ with $x \ne u$. We gather from the above that $xav \in XAY = A = uAv$. Therefore, $xav = ua'v$ for some $a' \in A$, and hence $xa = ua'$ (because $v$ is a unit and units are cancellative elements). It is then clear that $a \ne a'$; otherwise, $a'$ is a unit, and $xa = ua'$ yields $x = u$ (a contradiction). 

Set $B := A \setminus \{a'\}$. From the preceding considerations, it is
immediate that
\[
\begin{split}
A &= (A \setminus \{ua'v\}) \cup \{ua'v\}
   = (uAv \setminus \{ua'v\}) \cup \{xav\} \\
  &\subseteq uBv \cup xBv
   = \{u,x\}Bv
   \subseteq XAY = A.
\end{split}
\]
Hence $A = \{u,x\}Bv$, which implies that $A$ is not irreducible in
$\mathcal P_{\fin,\times}(H)$. 

Indeed, both $\{u,x\}$ and $Bv$ contain a unit of $H$
(in particular, $av \in Bv \cap H^\times$) and are therefore divisors
of $A$ in $\mathcal P_{\fin,\times}(H)$. Yet, by
Eq.~\eqref{equ:increasing-size}, neither $\{u,x\}$ nor $Bv$
is divisible by $A$, since $|\{u,x\}| = 2 < |A|$ and
$|Bv| = |B| = |A| - 1$.
\end{proof}

\begin{proposition}
\label{prop:atoms-of-P_fin_times}
If $H$ is Dedekind-finite, then every irreducible of $\mathcal P_{\fin,\times}(H)$ with at least three elements is an atom. 
\end{proposition}

\begin{proof}
Assume $A = XY$ for some $X, Y \in \mathcal P_{\fin,\times}(H)$. We need to show that $X$ or $Y$ is a unit. Since $A$ is irreducible, either $X$ or $Y$ is divisible by $A$. Suppose we are in the former case (the other being analogous). There then exist $S, T \in \mathcal P_{\fin,\times}(H)$ such that $X = SAT$ and hence $A = XY = SATY$. By Lemma \ref{lem:large-irreducibles-in-Pfin_times}, this is only possible if $|S| = |TY| = 1$, which, by Eq.~\eqref{equ:increasing-size}, yields $|Y| = 1$ and hence proves, by Theorem \ref{thm:divisor-closedness}\ref{thm:divisor-closedness(i)}, that $Y$ is a unit (as desired).
\end{proof}

\begin{proposition}
\label{prop:irreducibles-of-Pfin_times(1)}
Let $A$ be an irreducible of $\mathcal P_{\fin,\times}(H)$, and suppose that $1_H \in uAv$ for some $u, v \in H^\times$. If $H$ is Dedekind-finite, then $uAv$ is irreducible in $\mathcal P_{\fin,1}(H)$.
\end{proposition}

\begin{proof}
If $|A| = 2$, then $|uAv| = 2$, and we know from Proposition 
\ref{prop:two-element-irreds} that every two-element set in 
$\mathcal P_{\fin,1}(H)$ is irreducible.

Assume $|A| \ge 3$, and suppose that $H$ is Dedekind-finite. Then, by Proposition \ref{prop:atoms-of-P_fin_times}, 
$A$ is an atom of $\mathcal P_{\fin,\times}(H)$, and hence, by 
\cite[Proposition~3.1(iii)]{An-Tr18}, the same is true of $uAv$. 
Since $1_H \in uAv$ (by hypothesis), it follows that $uAv$ is also an atom 
(and therefore an irreducible) of $\mathcal P_{\fin,1}(H)$.
\end{proof}

\begin{proposition}
\label{prop:irreducibles-of-Pfin_times(2)}
If $H$ is Dedekind-finite, then $uAv$ is an irreducible of $\mathcal P_{\fin,\times}(H)$ for all $u, v \in H^\times$ and every irreducible $A$ of $\mathcal P_{\fin,1}(H)$.
\end{proposition}

\begin{proof}
Let $A$ be an irreducible of $\mathcal P_{\fin,1}(H)$, and suppose that $uAv = XY$ for some $u, v \in H^\times$ and $X, Y \in \mathcal P_{\fin,\times}(H)$. We claim that 
\[
uAv \text{ divides either }
X \text{ or } Y 
\text{ in }
\mathcal P_{\fin,\times}(H).
\]

To begin, $A = \allowbreak u^{-1}XYv^{-1}$, and thus
$1_H = u^{-1}xyv^{-1}$ for some $x \in X$ and $y \in Y$.  
Since $H$ is Dedekind-finite, $q := u^{-1}x \in u^{-1}X$ is a unit of $H$, with inverse 
$q^{-1} = \allowbreak yv^{-1} \in Yv^{-1}$. Consequently, we can write $A = X'Y'$, where
\[
X' := u^{-1}Xq^{-1} \in \mathcal P_{\fin,1}(H)
\quad\text{and}\quad
Y' := qYv^{-1} \in \mathcal P_{\fin,1}(H).
\]
By the hypothesis that $A$ is irreducible in $\mathcal P_{\fin,1}(H)$, either $X'$ or $Y'$ is divided by $A$ in the same monoid. Without loss of generality, suppose we are in the former case. There then exist $S, \allowbreak T \in \allowbreak \mathcal P_{\fin,1}(H)$ such that $SAT = X' = u^{-1}Xq^{-1}$, and hence
\[
X = (uSu^{-1})(uAv)(v^{-1}Tq).
\]
This shows that $uAv$ divides $X$ in 
$\mathcal P_{\fin,\times}(H)$, as desired. Since $|uAv| = |A| \ge 2$, 
it follows from Theorem \ref{thm:divisor-closedness}\ref{thm:divisor-closedness(i)} 
that $uAv$ is an irreducible element of $\mathcal P_{\fin,\times}(H)$.
\end{proof}

\begin{corollary}
If $H$ is Dedekind-finite, every irreducible of $\mathcal P_{\fin,\times}(H)$ is a quark.
\end{corollary}

\begin{proof}
This is a direct consequence of Propositions \ref{prop:atoms-are-quarks}, \ref{prop:two-element-irreds}, and \ref{prop:atoms-of-P_fin_times}.
\end{proof}

\subsection{Finale with lengths.} Another reason why most work on the arithmetic of (finitary) power monoids has focused so far on $\mathcal P_{\fin,1}(H)$ is that, when $H$ is Dedekind-finite, $\mathcal P_{\fin,1}(H)$ shares many arithmetic features with $\mathcal P_{\fin,\times}(H)$ and, by Theorem~\ref{thm:divisor-closedness}\ref{thm:divisor-closedness(ii)}, the latter is divisor-closed in $\mathcal P_{\fin}(H)$. Details are worked out below.

\begin{theorem}
If $H$ is Dedekind-finite, then $\mathcal P_{\mathrm{fin},1}(H)$ has the same system of lengths as $\mathcal P_{\fin,\times}(H)$.
\end{theorem}

\begin{proof}
Fix a non-unit $X \in \mathcal P_{\fin,\times}(H)$, pick  $u \in X \cap H^\times$, and set $X' := Xu^{-1}$. By Theorem \ref{thm:divisor-closedness}\ref{thm:divisor-closedness(i)}, $\{1_H\} \ne X' \in \mathcal P_{\fin,1}(H)$. Let $L$ be the length set of $X$ in
$\mathcal P_{\fin,\times}(H)$, and $L'$ be the length set of $X'$ in $\mathcal P_{\fin,1}(H)$. We claim that $L = L'$.

To begin, let $A_1 \ast \cdots \ast A_n$ be a length-$n$ factorization of $X$ in $\mathcal P_{\fin,\times}(H)$. Then, for each $i \in \llb 1, n \rrb$ there exists $u_i \in A_i \cap H^\times$ such that $u = u_1 \cdots u_n$; accordingly, define $v_i := u_1 \cdots u_i$ and $B_i := v_{i-1} A_i v_i^{-1}$, where $v_0 := 1_H$.  
We have from Proposition \ref{prop:irreducibles-of-Pfin_times(1)} that any $B_i$ is an irreducible of
$\mathcal P_{\fin,1}(H)$; in particular, note that 
\[
1_H = v_{i-1}u_iv_i^{-1} \in B_i.
\]
Since $B_1 \cdots B_n = v_0 A_1 \cdots A_n u^{-1} = X'$, it follows that $B_1 \ast  \allowbreak \cdots \ast \allowbreak B_n$ is a length-$n$ factorization of $X'$ in $\mathcal P_{\fin,1}(H)$, thereby yielding $L \subseteq L'$.

The reverse inclusion $L' \subseteq L$ is immediate: if
$C_1 \ast \cdots \ast C_m$ is a length-$m$ factorization of $X'$ in $\mathcal P_{\fin,1}(H)$,
then Proposition \ref{prop:irreducibles-of-Pfin_times(2)} implies that $C_1 \ast \cdots \ast C_{m-1} \ast (C_m u)$
is a length-$m$ factorization of $X$ in $\mathcal P_{\fin,\times}(H)$.

All in all, we see from the above that every
length set of $\mathcal P_{\fin,\times}(H)$ is also a
length set of $\mathcal P_{\fin,1}(H)$.
Moreover, applying the same argument with $u = 1_H$
under the additional assumption that $1_H \in X$,
we obtain that every length set of
$\mathcal P_{\fin,1}(H)$ occurs in
$\mathcal P_{\fin,\times}(H)$.
Thus, $\mathcal P_{\fin,1}(H)$ and
$\mathcal P_{\fin,\times}(H)$ have the same
systems of lengths. 
\end{proof}

With this point clarified, we know from Theorem~\ref{thm:ACC-and-Dedekind-finiteness}\ref{thm:ACC-and-Dedekind-finiteness(i)} that every $X \in \mathcal P_{\fin,1}(H)$ has a factorization (into irreducibles) regardless of any hypothesis on $H$. It is natural to ask
whether any more refined conclusions can be drawn about these factorizations.

\begin{theorem}
\label{thm:BF-iff-FF-iff-atomic}
The following conditions are equivalent:
\begin{enumerate}[label=\textup{(\alph{*})}]
\item\label{thm:BF-iff-FF-iff-atomic(a)} $H$ is torsion-free.

\item\label{thm:BF-iff-FF-iff-atomic(b)} $\mathcal P_{\fin,1}(H)$ is an FF monoid.

\item\label{thm:BF-iff-FF-iff-atomic(c)} $\mathcal P_{\fin,1}(H)$ is a BF monoid.
\end{enumerate}
\end{theorem}

\begin{proof}
The implication \ref{thm:BF-iff-FF-iff-atomic(a)} $\Rightarrow$ 
\ref{thm:BF-iff-FF-iff-atomic(c)} is straightforward from \cite[Theorem 3.11(i)]{An-Tr18} and Corollary \ref{cor:atomicity}. 
As for the dual implication \ref{thm:BF-iff-FF-iff-atomic(c)} $\Rightarrow$ 
\ref{thm:BF-iff-FF-iff-atomic(a)}, it is enough to consider that if $H$ has a non-identity element $x$ of finite order $n$, then 
\[
X := \{1_H, x\}^n = \{1_H, x\}^{n+k},
\qquad\text{for all } 
k \in \mathbb N.
\]
By Proposition \ref{prop:two-element-irreds}, this shows that the (infinite) interval $\llb n, \infty\rrb$ is contained in the length set of $X$, and hence $\mathcal P_{\fin,1}(H)$ is not BF.

The equivalence \ref{thm:BF-iff-FF-iff-atomic(b)} $\Leftrightarrow$ 
\ref{thm:BF-iff-FF-iff-atomic(c)} is, on the other hand, a special case of 
\cite[Theorem 4.11(i)]{Co-Tr-22(a)}. It suffices to note that, by 
Eq.~\eqref{eq:to-divide-implies-to-be-contained}, every 
$X \in \mathcal P_{\fin,1}(H)$ has only finitely many divisors (in particular, it is divided by finitely many irreducibles), and hence 
$\mathcal P_{\fin,1}(H)$ is locally of finite type in the sense of 
\cite[Definition 4.1(3)]{Co-Tr-22(a)}. For additional details, see \cite[Example 4.5(3)]{Co-Tr-22(a)}.
\end{proof}

\begin{corollary}
The following conditions are equivalent for a monoid $H$:
\begin{enumerate}[label=\textup{(\alph{*})}]
\item\label{cor:UF-iff-HF-iff-trivial(1)} $H$ is trivial.

\item\label{cor:UF-iff-HF-iff-trivial(2)} $\mathcal P_{\fin,1}(H)$ is a UF monoid.

\item\label{cor:UF-iff-HF-iff-trivial(3)} $\mathcal P_{\fin,1}(H)$ is an HF monoid.
\end{enumerate}
\end{corollary}

\begin{proof}
We show that \ref{cor:UF-iff-HF-iff-trivial(3)} $\Rightarrow$ \ref{cor:UF-iff-HF-iff-trivial(1)}, as the implications \ref{cor:UF-iff-HF-iff-trivial(1)} $\Rightarrow$ \ref{cor:UF-iff-HF-iff-trivial(2)} $\Rightarrow$ \ref{cor:UF-iff-HF-iff-trivial(3)} are obvious. 

Assume that $\mathcal P_{\fin,1}(H)$ is HF. Since every HF monoid is BF, we have from Theorem~\ref{thm:BF-iff-FF-iff-atomic} that $H$ is torsion-free. 
Suppose, for a contradiction, that $H$ is non-trivial, and let $x$ be a 
non-identity element of $H$. Then, by Proposition \ref{prop:two-element-irreds}, the length set of
\[
X := \{1_H, x\}^3 = \{1_H, x\}\{1_H, x^2\}
\]
is not a singleton. This, however, contradicts the hypothesis that 
$\mathcal P_{\fin,1}(H)$ is HF.
\end{proof}

We finish with some results concerning minimal factorizations, starting with a generalization of  \cite[Proposition 4.12(i)]{An-Tr18} from torsion-free to arbitrary monoids and from atomic factorizations to factorizations into irreducibles (the point here is that, if $H$ is torsion-free, then by Theorem \ref{thm:Pfin1-is-factorable}\ref{thm:Pfin1-is-factorable(ii)} every irreducible of $\mathcal P_{\fin,1}(H)$ is an atom).

\begin{theorem}\label{thm:Pfin1-FmF}
The largest minimal length of a set $X \in \mathcal P_{\fin,1}(H)$ is smaller than or equal to $|X|-1$. In particular, $\mathcal P_{\fin,1}(H)$ is BmF and FmF.
\end{theorem}

\begin{proof}
The equivalence between BmFness and FmFness is a direct consequence of \cite[Theorem 4.11(i)]{Co-Tr-22(a)}, as in the proof of the equivalence between items \ref{thm:BF-iff-FF-iff-atomic(b)} and \ref{thm:BF-iff-FF-iff-atomic(c)} in Theorem \ref{thm:BF-iff-FF-iff-atomic}. Therefore, we focus below on the first part of the statement.

The claim is trivial if $X = \{1_H\}$, when the only factorization of $X$ is the empty word; or if $X$ is irreducible, in which case $|X|$ is at least $2$ and, by Theorem \ref{thm:Pfin1-is-factorable}\ref{thm:Pfin1-is-factorable(ii)}, $X$ has a unique minimal factorization of length one. So, assume henceforth that $X$ is neither the identity nor an irreducible of $\mathcal P_{\fin,1}(H)$, and let $\mathfrak a = A_1 \ast \cdots \ast A_n$ be a minimal factorization of $X$ (into irreducibles). Note that $n \ge 2$. We claim that
\begin{equation}\label{equ:cascading_containment}
	 A_1\cdots A_i \subsetneq A_1\cdots A_{i+1}, \quad \text{for every }i \in \llb 1, n-1 \rrb.
\end{equation}
Indeed, assume the contrary and denote by $\varepsilon$ the empty word. There then exist an index $i \in \llb 1, n-1 \rrb$ such that $A_1\cdots A_i = A_1\cdots A_{i+1}$. Accordingly, set $\mathfrak b := \varepsilon$ if $i = \allowbreak n-1$, and $\mathfrak b := A_{i+2} \ast \cdots \ast A_n$ otherwise. Then $A_1 \ast \cdots \ast A_i \ast \mathfrak b$ is a \textit{proper} subword of $\mathfrak a$ as well as a factorization of $X$, contradicting the minimality of $\mathfrak a$.

Consequently, Eq.~\eqref{equ:cascading_containment} holds, and hence $2 \le |A_1\cdots A_i | < |A_1\cdots A_{i+1}| \le |X|$ for all $i \in \llb 1, n-1 \rrb$, which implies at once that $n\le |X|-1$.
\end{proof}

For the next theorem, we recall that a semigroup $S$ is \evid{breakable} if $xy \in \{x, y\}$ for all $x, y \in S$. This is a rather special class of semigroups, first introduced by R\'edei \cite[Section 27]{Red-1967}, with a rich structural theory.

\begin{theorem}
If $H$ is commutative, then the following conditions are equivalent:
\begin{enumerate}[label=\textup{(\alph{*})}]
\item $\mathcal P_{\fin,1}(H)$ is UmF.
\item $H \setminus H^\times$ is a breakable subsemigroup of $H$, the unit group of $H$ is either trivial or cyclic of order two, and $ux = x$ for all $u \in H^\times$ and $x \in H \setminus H^\times$.
\end{enumerate}
\end{theorem}

\begin{proof}
See Theorem 4.8 in \cite{Co-Tr-25(a)}.
\end{proof}

To date, there is no satisfactory characterization of when 
$\mathcal P_{\fin,1}(H)$ is UmF in the non-commutative case, and the situation is even more dire for HmFness. 
For partial progress, see Lemma 3.4(i), Theorem 3.6, and Proposition 4.5 in \cite{Co-Tr-25(a)}. By contrast, the situation is completely settled for atomic factorizations.

\begin{theorem}
\label{thm:PMs-minimal-factorizations}
The following conditions hold:
\begin{enumerate}[label=\textup{(\roman{*})}]
\item\label{thm:PMs-minimal-factorizations(i)} $\mathcal P_{\fin,1}(H)$ is UmF-atomic if and only if $H$ is trivial.

\item\label{thm:PMs-minimal-factorizations(ii)} $\mathcal P_{\fin,1}(H)$ is HmF-atomic if and only if $H$ is trivial or a cyclic group of order $3$.
\end{enumerate}
\end{theorem}

\begin{proof}
These are Theorem 4.14 and Corollary 4.15 of \cite{An-Tr18}, respectively.
\end{proof}
\section{Into the future.}
\label{sec:future}

In this final section, we present a selection of open problems of an arithmetic
flavor, which we hope will stimulate further research on power monoids (and,
more generally, on power semigroups) in the years to come.

A monoid is \evid{torsion} if each of its elements has finite order (note that a torsion-free monoid is not a monoid that is not torsion). We begin with the following:

\begin{conjecture}\label{conj:systems-of-lengths}
If $H$ is not a torsion monoid, then every non-empty finite set of integers greater than one can be realized as the length set of a set $X \in \mathcal P_{\fin,1}(H)$.
\end{conjecture}

The conjecture was first stated in \cite[Sect.~5]{Fa-Tr18} for the special case 
$H = \mathbb N$, which, however, turns out to be equivalent to the general case. 
Indeed, any monoid $H$ containing an element of infinite order has a submonoid 
isomorphic to $\mathbb N$. By Proposition 
\ref{prop:divisor-closed-submons} and 
Theorem \ref{thm:ACC-and-Dedekind-finiteness}\ref{thm:ACC-and-Dedekind-finiteness(ii)}, 
it follows that $\mathscr L(\mathcal P_{\fin,0}(\mathbb N))$ is contained in $\mathscr L(\mathcal P_{\fin,1}(H))$, i.e., every length set of $\mathcal P_{\fin,0}(\mathbb N)$ is also a length set of $\mathcal P_{\fin,1}(H)$.

Now, we gather from Theorems~\ref{thm:Pfin1-is-factorable}\ref{thm:Pfin1-is-factorable(ii)} and \ref{thm:BF-iff-FF-iff-atomic} that the monoid 
$\mathcal P_{\fin,0}(\mathbb N)$ is BF and its irreducibles are all atoms. 
Hence, the conjecture reduces to proving that the system of lengths of 
$\mathcal{P}_{\text{fin},0}(\mathbb{N})$ consists of the singletons $\{0\}$ and $\{1\}$, along with all non-empty finite sets of integers greater than one.

As for partial progress, it is known from \cite[Propositions~4.8--4.10]{Fa-Tr18} 
that, for each $n \ge 2$, the singleton $\{n\}$, the interval $\llb 2, n \rrb$, 
and the set $\{2, n\}$ are length sets of $\mathcal P_{\fin,0}(\mathbb N)$. 
This remained the state of the art until recently. A breakthrough came in summer 2025, when Reinhart \cite{Rei-2026} 
proved that $\mathcal P_{\fin,0}(\mathbb N)$ is \evid{fully elastic}: 
for every rational $q \ge 1$, there exists a length set $L$ of 
$\mathcal P_{\fin,0}(\mathbb N)$ such that $\max L = q \min L$.

We continue with a question reminiscent of the ``Characterization Problem for Systems of Lengths'' in the classical theory of factorizations \cite[Section 6]{Ger-2016}.

\begin{question}[Characterization Problem for Power Monoids]
\label{question:characterization-problem}
Given a class $\mathscr C$ of \textit{finite} monoids (for instance, the class of
finite groups), determine whether there exists an integer $n \ge 1$ such that if $H, K \in \mathscr C$ have the property that
(i) the systems of \textit{minimal} lengths of $\mathcal{P}_{\fin,1}(H)$ and
$\mathcal{P}_{\fin,1}(K)$ coincide, and (ii) at least one of $H$ and $K$ has
order greater than or equal to $n$, we have 
$H \cong K$.
\end{question}

The question is presented here for the first time. The answer is
negative for the class of \textit{all} finite monoids. Indeed, isomorphic monoids have the same
systems of minimal lengths, and we gather from
\cite[Example 1.2(3)]{Tri-Yan-23(a)} that, for every $n \ge 1$, there exist
non-isomorphic monoids $H$ and $K$ with $|H| = |K| = n$ whose reduced
finitary power monoids are isomorphic. This naturally links
Question~\ref{question:characterization-problem} to
Question~\ref{ques:BG-like-for-monoids}.

We conjecture that the answer to Question~\ref{question:characterization-problem} is positive for finite groups, for which it is known that Question~\ref{ques:BG-like-for-monoids} has a positive answer \cite{Tri-Yan2026(a)}. Partial results in the cyclic case appear in \cite[Section~5]{An-Tr18}, and there are a few more observations worth recording.

First, Conjecture \ref{conj:systems-of-lengths} and Theorem \ref{thm:BF-iff-FF-iff-atomic} imply that any two non-trivial torsion-free monoids have the same system of lengths, consisting of $\{0\}$, $\{1\}$, and each non-empty finite subset of $\llb 2, \infty \rrb$ (cf.~the comments under Conjecture \ref{conj:systems-of-lengths}). On the other hand, we have from \cite[Example 2.2(1)]{Co-Tr-22(a)} and \cite[Proposition 4.7(v)]{An-Tr18} that every factorization in a unit-cancellative commutative monoid is minimal, and there are plenty of non-isomorphic torsion-free commutative monoids;
for instance, it is folklore that two numerical monoids are isomorphic if and only if they are equal \cite[Theorem 3]{Hig-1969}. All in all, this is part of the reason why Question \ref{question:characterization-problem} is restricted to \textit{finite} monoids.

Second, it follows from Proposition \ref{prop:divisor-closed-submons} and Theorem \ref{thm:ACC-and-Dedekind-finiteness}\ref{thm:ACC-and-Dedekind-finiteness(ii)}
that the system of minimal lengths of a monoid $H$ is contained in the system of minimal lengths of $\mathcal{P}_{\fin,1}(K)$ for every monoid $K$ into which $H$ embeds.

Third, Theorem \ref{thm:Pfin1-FmF} shows that, for every monoid $H$, the largest minimal length $\mu_H(X)$ of a set $X$ in $\mathcal{P}_{\fin,1}(H)$ is bounded above by $|X|-1$. When $H$ is finite, it follows that the system of minimal lengths of $\mathcal P_{\fin,1}(H)$ lies in the power set of the interval $\llb 0, n - 1\rrb$, where $n := |H|$. Moreover, $H$ itself is an element of $\mathcal P_{\fin,1}(H)$, and it is not difficult to establish that 
\begin{equation}\label{equ:log-inequality}
\mu_H(H) \ge \log_2(n).
\end{equation}

Indeed, assume $n \ge 2$ (otherwise the claim is obvious), and let $x_1, \ldots, x_{n-1}$ be an enumeration of the non-identity elements of $H$. By Proposition \ref{prop:two-element-irreds}, the word 
\[
\mathfrak a := \{1_H, x_1\} \ast \cdots \ast \{1_H, x_{n-1}\}
\]
is a factorization (into irreducibles) of $H$ in $\mathcal P_{\fin,1}(H)$. On the other hand, if $\sigma$ is an injective map $\llb 1, \allowbreak k \rrb \to \allowbreak \llb 1, n -1 \rrb$ for some $k \ge 1$, then for the word 
\[
\{1_H, \allowbreak x_{\sigma(1)}\} \ast \cdots \ast \{1_H, x_{\sigma(k)}\}
\]
to be a factorization of $H$ it is necessary that $2^k \ge n$, i.e., $k \ge \log_2(n)$. Since every factorization of $H$ contains a subword that, up to a permutation of its letters, is a \textit{minimal} factorization of $H$, this entails the desired inequality on $\mu_H(H)$.

In particular, Eq.~\eqref{equ:log-inequality} implies that, given a finite monoid $H$, there are only finitely many monoids $K$ (up to isomorphism) such that the systems of 
minimal lengths of $\mathcal{P}_{\text{fin},1}(H)$ and $\mathcal{P}_{\text{fin},1}(K)$ coincide. This relies on the fact that, up to isomorphism, there are finitely many 
monoids of order less than or equal to a fixed integer $N \ge 1$.

\begin{question}
Given a monoid $H$ and an integer $k \ge 1$, denote by
$\mathscr U_k(\mathcal P_{\fin,1}(H))$ and $\mathscr U_k^{\sf m}(\mathcal P_{\fin,1}(H))$ the union of all length sets and the union of all minimal length sets of $\mathcal P_{\fin,1}(H)$ containing $k$, respectively.
Is $\mathscr U_k(\mathcal P_{\fin,1}(H))$ an interval for every [large] $k \in \mathbb N^+$? What about
$\mathscr U_k^{\sf m}(\mathcal P_{\fin,1}(H))$? Note that the empty set
is an interval.
\end{question}

As a starting point, let $\mathbb Z_n$ be the additive group of integers modulo $n$, where $n \ge 2$. In \cite[Example 3.12(2)]{Co-Tr-22(b)}, Cossu and Tringali noted that
\[
\mathscr U_k(\mathcal P_{\fin,0}(\mathbb Z_n)) = 
\left\{
\begin{array}{ll}
\{1\} & \text{if } k = 1 \text{ and } n \text{ is odd}, \\
\mathbb N^+ & \text{if } k = 1 \text{ and } n \text{ is even}, \\
\llb 2, \infty \rrb & \text{if } k \ge 2,
\end{array}
\right.
\]
and
\[
\mathscr U_k^{\sf m}(\mathcal P_{\fin,0}(\mathbb Z_n)) =
\left\{
\begin{array}{ll}
\llb 2, n-1 \rrb & \text{if } 2 \le k < n, \\
\emptyset & \text{if }n \le k.
\end{array}
\right.
\]
In addition, Theorem \ref{thm:Pfin1-is-factorable}\ref{thm:Pfin1-is-factorable(ii)} implies that $\mathscr U_1^{\sf m}(\mathcal P_{\fin,0}(\mathbb Z_n)) = \{1\}$.

On the other hand, we have already observed in the above that if $H$ contains an element of infinite order, then $\mathcal P_{\fin,1}(H)$ has a divisor-closed submonoid isomorphic to $\mathcal P_{\fin,0}(\mathbb N)$. By \cite[Proposition~4.8]{Fa-Tr18}, this implies that
\[
\llb 2, \infty \rrb \subseteq \mathscr U_k(\mathcal P_{\fin,1}(H)),
\qquad \text{for every integer }
k \ge 2.
\]

To complete our list of problems, we now turn to one of the most intriguing features of power semigroups: the interplay between arithmetic and combinatorial properties. 

As a concrete illustration of this phenomenon (beyond the conjectures of S\'ark\"ozy and Ostmann already mentioned in Section~\ref{sec:01}), let $H$ be a numerical monoid, and for all $k, l \in \mathbb{N}$ let $\alpha_{k,l}(H)$ denote the number of $k$-element atoms of $\mathcal{P}_{\fin,0}(H)$ whose maximum is less than or equal to $l$. Experiments appear to support the following:

\begin{conjecture}[Unimodality Conjecture]
\label{conj:unimodality}
For each $l \in \mathbb N$, there exists an index $k \in \allowbreak \llb 0, l \rrb$ such that $\alpha_{0, l} \le \cdots \le \alpha_{k,l}$ and $\alpha_{k,l} \ge \cdots \ge \alpha_{l,l}$. 
\end{conjecture}

Prior to the present work, the conjecture had only circulated in the special case $H = \mathbb{N}$. Originally proposed by the author in an email to Geroldinger and Gotti dated January~20, 2023, it was later included as Problem~1 in the problem set of the online collaborative project ``\textsc{CrowdMath} 2023: Arithmetic of Power Monoids'', hosted throughout 2023 on the Art of Problem Solving (AoPS) website at
\begin{center}
\url{https://artofproblemsolving.com/polymath/mitprimes2023}
\end{center}

As a non-trivial first step, Aggarwal et al.~\cite[Corollary~5.3]{Agg-Got-Su-2024} have shown by a probabilistic approach that, again in the case $H = \mathbb{N}$, the conjecture holds for almost all $l \in \mathbb{N}$, where ``almost all'' is understood in the sense of asymptotic density.

\section{Acknowledgments.}

I was supported by the Natural Science Foundation of Hebei Province through grant A2023205045. 
I thank Federico Campanini (UCLouvain, Belgium), Laura Cossu (University of Cagliari, Italy), and Alfred Geroldinger (UniGraz, Austria) for numerous comments that have significantly improved the exposition. I am particularly grateful to Balint Rago (UniGraz, Austria) and an anonymous referee for a very thorough proofreading of an earlier version.

\begin{biog}
\item[Salvatore ``Salvo'' Tringali] is a full professor in the School of Mathematical Sciences at Hebei Normal University (Shijiazhuang, China). His research lies at the intersection of arithmetic combinatorics and semigroup theory, with particular emphasis on power semigroups and non‑unique factorization phenomena in algebraic structures such as monoids and rings.

He holds a Ph.D.~in mathematics from the University of Lyon, Jean Monnet campus (St-\'Etienne, France), and a Ph.D.~in electronic engineering from Universit\`a Mediterranea (Reggio Cal., Italy). Prior to his current appointment, he was a visiting researcher at the Center for Combinatorics, Nankai University (Tianjin, China), a Lise Meitner fellow and lecturer at the University of Graz (Graz, Austria), a postdoc at the Centre de ma\-th\'e\-ma\-tiques Laurent Schwartz, \'Ecole polytechnique (Palaiseau, France), a research associate at Texas A\&M University at Qatar (Doha, Qatar), and a Marie Curie fellow at the Laboratoire Jacques‑Louis Lions, Sorbonne University (Paris, France).
\begin{affil}

\vskip 0.2cm

School of Mathematical Sciences, Hebei Normal University \\ Shijiazhuang, Hebei province, 050024 China \\
salvo.tringali@gmail.com
\end{affil}
\end{biog}

\vfill\eject

\end{document}